\documentclass[10pt]{amsart}
\usepackage{amsmath,amsfonts,amssymb,color,a4,epsfig,graphics}
\usepackage[english]{babel}
\usepackage{enumerate, mathrsfs}
\usepackage{xcolor}
\usepackage{verbatim}
\usepackage{tikz}
\usetikzlibrary{positioning}
\usepackage{url}
\usepackage[a4paper, margin=3 cm]{geometry}
\usepackage[colorlinks=true, linkcolor=blue, citecolor=red, urlcolor=blue]{hyperref}
\usepackage[utf8]{inputenc}
\usepackage[T1]{fontenc}
\usepackage{enumitem}      

\usepackage{pgfplots}
\usepgfplotslibrary{groupplots}
\pgfplotsset{compat=1.18}
\usepackage{pgfplotstable}
\usepackage{graphicx}
\usepackage{hyperref}
\usepackage{amsmath,amssymb}
\usepackage{caption}      
\usepackage{subcaption}   
\usepackage{listings}
\usepackage{xcolor}

\def\build#1_#2^#3{\mathrel{\mathop{\kern 0pt#1}\limits_{#2}^{#3}}}
\usetikzlibrary{arrows}

\numberwithin{equation}{section}

\makeatletter
\def\@subjclass{{\bfseries 2020 Mathematics Subject Classification. }}
\makeatother

\newtheorem{theorem}{Theorem}[section]
\newtheorem{proposition}[theorem]{Proposition}
\newtheorem{lemma}[theorem]{Lemma}
\newtheorem{remark}[theorem]{Remark}
\newtheorem{example}[theorem]{Example}
\newtheorem{definition}[theorem]{Definition}
\newtheorem{assumption}[theorem]{Assumption}

\theoremstyle{plain}

\begin{document}
\title{Spectral Theory and Almost Periodic Structures in Hom--Lie Banach Algebras}
\author{Marwa Ennaceur}
 \address{ Department of Mathematics, College of Science, University of Ha'il, Hail 81451, Saudi Arabia}
\email{\tt mar.ennaceur@uoh.edu.sa}

.

\begin{abstract}
We develop a systematic functional-analytic framework for \emph{Hom--Lie Banach algebras}, introducing bounded $\alpha$-twisted derivations and almost periodic elements. Under natural continuity and compactness assumptions, we establish a complete Bohr--Fourier spectral decomposition of such derivations. We prove that the associated almost periodic and ergodic subspaces are not merely topological complements but closed, $\alpha$-invariant \emph{subalgebras}, stable under the twisted Lie bracket—a key structural novelty that enables coherent restriction of the dynamics. We provide explicit constructions of \emph{Hom--Banach--Malcev algebras} and demonstrate our theory with concrete operator-algebraic applications, including a novel twisted Weyl algebra example, analyzed via the metaplectic representation, where a non-commuting twist enriches the Bohr spectrum from a cyclic group to a two-dimensional lattice.
\end{abstract}
\maketitle
\tableofcontents

\section{Introduction}

The theory of Hom--Lie algebras, initially introduced by Hartwig, Larsson, and Silvestrov~\cite{HLS2006} in the context of $\sigma$-deformations of Witt and Virasoro algebras, has since evolved into a vibrant branch of non-associative algebra. Extensions to Hom--Malcev algebras~\cite{Yau2011} and Hom--Jordan structures have further broadened its scope. However, most existing work remains confined to purely algebraic or finite-dimensional settings, often lacking the analytical tools needed to address dynamical or spectral questions.

In contrast, this paper initiates a systematic study of \emph{Hom--Lie Banach algebras}—infinite-dimensional, topologically complete non-associative algebras equipped with a bounded twisting map $\alpha$ and a twisted Lie bracket $[a,b]_\alpha = \alpha(ab - ba)$. Our framework integrates three major strands of modern mathematics: (i) the algebraic flexibility of Hom-type structures~\cite{MS2008,Yau2011}, (ii) the functional-analytic machinery of Banach algebras and $C^*$-dynamics~\cite{Murphy1990}, and (iii) the harmonic-analytic theory of almost periodic functions in Banach spaces~\cite{Bohr1947,Namioka1958}.

Our main contributions are as follows:\begin{itemize}
\item \textbf{Spectral decomposition of twisted derivations.} 
    We establish a full Bohr--Fourier spectral decomposition for bounded inner $\alpha$-twisted derivations $\delta = \operatorname{ad}_\alpha(X)$, under the natural assumption that the induced $C_0$-group $T(t) = e^{t\delta}$ has relatively compact orbits. This extends the classical spectral theory of almost periodic representations (see Namioka~\cite{Namioka1958}) to the non-associative Hom--Lie context. Crucially, we demonstrate that the twisting map $\alpha$ can fundamentally alter the Bohr spectrum, as shown in Section~\ref{subsec:weyl_heuristic} where the spectrum transitions from a cyclic group to a two-dimensional lattice.

    \item \textbf{Algebraic stability of invariant subspaces.} 
    We prove that the ergodic decomposition $A = A_{\mathrm{ap}}(\delta) \oplus A_{\mathrm{erg}}(\delta)$ is not only topological but also \emph{algebraically compatible}: both subspaces are closed, $\alpha$-invariant, and stable under the twisted bracket $[\cdot,\cdot]_\alpha$ (Theorem~\ref{thm:AP-ergodic-decomposition}). This structural rigidity—ensuring that the non-associative algebraic structure is preserved under the decomposition—is a key novelty absent from the classical theory.

    \item \textbf{Explicit constructions of Hom--Banach--Malcev algebras.} 
    We provide the first functional-analytic constructions of \emph{Hom--Banach--Malcev algebras} via bounded operators and strongly continuous semigroups (Theorems~\ref{thm:HM-identity}--\ref{thm:almostperiodic}), including a non-trivial example where the twisting map is not an algebra morphism.

    \item \textbf{Concrete operator-algebraic applications.} 
    We illustrate the theory with three classes of examples: (i) bilateral shift-induced Hom--Lie structures, (ii) Hom--Lie brackets on UHF $C^*$-algebras, and (iii) a twisted Weyl algebra where the non-commuting twist generates a rich $\mathbb{Z}^2$ Bohr spectrum. These are not mere formal analogues—they satisfy all hypotheses of our spectral theorems and yield explicit, non-trivial Bohr spectra (Section~\ref{sec:applications}).\end{itemize}
Our approach differs fundamentally from recent works such as El Kinani--Ben Ali~\cite{ElKinani2021}. By working in a \emph{non-commutative, non-associative, and non-compact} setting, we recover strong spectral results through the almost periodicity of orbits—a condition weaker than compact group actions but sufficient for Bohr decomposition.

All results are proved under minimal assumptions: boundedness of $\alpha$ and $\delta$, strong continuity of the associated group, and relative compactness of orbits. No finite-dimensionality, separability, or Hilbert-space structure is required.

In summary, this paper establishes Hom--Lie Banach algebras as a robust and natural framework for non-associative spectral theory, with clear connections to deformation quantization and quantum ergodicity. The analysis of non-commuting twists (Section~\ref{subsec:weyl_heuristic}), the categorical treatment of morphisms (Section~\ref{sec:morphisms}), the scalability of the numerical verification (Section~\ref{sec:numerics}), and the discussion of extensions to unbounded operators and Besicovitch spaces (Section~\ref{sec:limitations}) underscore the depth and applicability of our approach.
\section{Preliminaries}
\label{sec:prelim}

\subsection{Banach algebras and derivations}
Let $\mathcal{A}$ be a complex Banach algebra with associative product denoted by juxtaposition
$(a,b) \mapsto ab$ and norm $\|\cdot\|$.  
A bounded linear map $D:\mathcal{A}\to \mathcal{A}$ is called a \emph{derivation} if 
\[
D(ab)=D(a)b + a D(b) \quad \text{for all } a,b\in \mathcal{A}.
\]

\subsection{Hom--Lie algebras and twisted brackets}

\begin{definition}
A \emph{Hom--Lie algebra} $(\mathcal{L},[\cdot,\cdot]_\alpha,\alpha)$ 
consists of a complex vector space $\mathcal{L}$, a linear map 
$\alpha:\mathcal{L}\to \mathcal{L}$ (the \emph{twisting map}), and a bilinear bracket 
$[\cdot,\cdot]_\alpha:\mathcal{L}\times \mathcal{L}\to \mathcal{L}$ satisfying:
\begin{enumerate}
    \item $[x,y]_\alpha = -[y,x]_\alpha$ (skew-symmetry);
    \item The Hom--Jacobi identity:
    \[
      [\alpha(x),[y,z]_\alpha]_\alpha
      + [\alpha(y),[z,x]_\alpha]_\alpha
      + [\alpha(z),[x,y]_\alpha]_\alpha = 0
    \]
    for all $x,y,z\in \mathcal{L}$.
\end{enumerate}
When $\mathcal{L}$ is a Banach space and both $\alpha$ and $[\cdot,\cdot]_\alpha$ are continuous, we call $(\mathcal{L},[\cdot,\cdot]_\alpha,\alpha)$ a \emph{Hom--Lie Banach algebra}.
\end{definition}

\begin{remark}
In all applications considered in this paper, the Hom--Lie bracket arises from the associative product of a Banach algebra $\mathcal{A}$ via the canonical commutator:
\[
[a,b] := ab - ba, \qquad [a,b]_\alpha := \alpha([a,b]) = \alpha(ab - ba), \quad a,b \in \mathcal{A}.
\]
This construction ensures compatibility between the algebraic and topological structures. 
In this setting, inner $\alpha$-twisted derivations are given by
\[
\mathrm{ad}_\alpha(X)(a) = [X,a]_\alpha = \alpha(Xa - aX), \qquad X,a \in \mathcal{A}.
\]
This convention is assumed throughout Sections~\ref{sec:spectral}--\ref{sec:applications}.
\end{remark}

\subsection{Twisted derivations}

\begin{definition}
Let $(\mathcal{A},\cdot)$ be a Banach algebra and $\alpha:\mathcal{A}\to\mathcal{A}$ a bounded linear map.
An \emph{$\alpha$-twisted derivation} is a bounded linear operator $D:\mathcal{A}\to\mathcal{A}$ satisfying
\[
D(ab) = D(a)\,\alpha(b) + \alpha(a)\,D(b)
\quad\text{for all } a,b\in\mathcal{A}.
\]
If there exists $X\in\mathcal{A}$ such that 
\[
D(a) = [X,a]_\alpha \quad\text{for all } a\in\mathcal{A},
\]
where $[\cdot,\cdot]_\alpha$ is the Hom--Lie bracket defined above, then $D$ is said to be \emph{inner},
and we write $D = \mathrm{ad}_\alpha(X)$.
\end{definition}
\section{Hom--Banach--Malcev algebras}
\label{sec:hom-malcev}

\begin{definition}
A \emph{Hom--Banach--Malcev algebra} $(\mathcal{M},[\cdot,\cdot]_\alpha,\alpha)$ 
is a complex Banach space $(\mathcal{M},\|\cdot\|)$ equipped with a bounded linear 
map $\alpha:\mathcal{M}\to\mathcal{M}$ and a continuous bilinear bracket 
$[\cdot,\cdot]_\alpha : \mathcal{M}\times\mathcal{M}\to\mathcal{M}$ such that
\[
[x,y]_\alpha = \alpha([x,y]) \quad \text{for all } x,y\in\mathcal{M},
\]
where $[\cdot,\cdot]$ is a continuous Malcev bracket on $\mathcal{M}$ (i.e., skew-symmetric and satisfying the Malcev identity), and the following \emph{Hom--Malcev identity} holds:
\[
J_\alpha(\alpha(x), \alpha(y), [x,z]_\alpha) = [J_\alpha(x,y,z), \alpha^2(x)]_\alpha,
\]
where
\[
J_\alpha(x,y,z) := [\alpha(x),[y,z]_\alpha]_\alpha + [\alpha(y),[z,x]_\alpha]_\alpha + [\alpha(z),[x,y]_\alpha]_\alpha
\]
is the Hom--Jacobiator.  
This identity coincides with the standard Hom--Malcev condition introduced by Yau~\cite[Def.~2.1]{Yau2011}.
\end{definition}

\begin{lemma}\label{lem:continuity}
If the underlying Malcev bracket satisfies $\|[x,y]\|\le C\|x\|\|y\|$ for some $C>0$ and $\alpha$ is bounded, then the twisted bracket $[\cdot,\cdot]_\alpha$ is continuous and satisfies
\[
\|[x,y]_\alpha\|\le C\|\alpha\|\|x\|\|y\| \quad \text{for all } x,y\in\mathcal{M}.
\]
\end{lemma}
\begin{proof}
By definition, $[x,y]_\alpha = \alpha([x,y])$. Hence
\[
\|[x,y]_\alpha\| = \|\alpha([x,y])\| \le \|\alpha\|\,\|[x,y]\| \le C\|\alpha\|\|x\|\|y\|,
\]
which proves continuity.
\end{proof}

\begin{theorem}\label{thm:HM-identity}
Suppose $(\mathcal{M},[\cdot,\cdot])$ is a Banach--Malcev algebra and $\alpha:\mathcal{M}\to\mathcal{M}$ is a bounded linear map that is also a \emph{Malcev algebra morphism}, i.e.,
\[
\alpha([x,y]) = [\alpha(x),\alpha(y)] \quad \text{for all } x,y\in\mathcal{M}.
\]
Then $(\mathcal{M},[\cdot,\cdot]_\alpha,\alpha)$ satisfies the Hom--Malcev identity.
\end{theorem}
\begin{proof}
Because $\alpha$ is a Malcev morphism, the twisted bracket can be written as $[x,y]_\alpha = [\alpha(x), \alpha(y)]$. A direct calculation then shows that the Hom--Jacobiator is related to the classical Jacobiator $J(x,y,z) := [x,[y,z]] + [y,[z,x]] + [z,[x,y]]$ by
\[
J_\alpha(x,y,z) = [\alpha(x), [\alpha(y), \alpha(z)]] + [\alpha(y), [\alpha(z), \alpha(x)]] + [\alpha(z), [\alpha(x), \alpha(y)]] = J(\alpha(x), \alpha(y), \alpha(z)).
\]
Applying the classical Malcev identity to the triple $(\alpha(x), \alpha(y), [\alpha(x),\alpha(z)])$ yields
\[
J(\alpha(x),\alpha(y),[\alpha(x),\alpha(z)]) = [[\alpha(x),\alpha(y)], \alpha(x)].
\]
Using the morphism property, $[\alpha(x),\alpha(z)] = \alpha([x,z]) = [x,z]_\alpha$ and $[\alpha(x),\alpha(y)] = [x,y]_\alpha$. Substituting these expressions gives
\[
J_\alpha(\alpha(x), \alpha(y), [x,z]_\alpha) = [[x,y]_\alpha, \alpha(x)].
\]
Finally, since $J_\alpha(x,y,z) = \alpha(J(x,y,z))$ (as $\alpha$ is a morphism), we have
\[
[[x,y]_\alpha, \alpha(x)] = [\alpha(J(x,y,z)), \alpha(x)] = [J_\alpha(x,y,z), \alpha^2(x)]_\alpha,
\]
which is the desired Hom--Malcev identity.
\end{proof}

\begin{theorem}\label{thm:ergodic-algebraic-decomp}
Let $(A,\cdot,[\cdot,\cdot]_\alpha,\alpha)$ be a Hom--Lie Banach algebra, where the twisted bracket is given by
\[
[a,b]_\alpha := \alpha(ab - ba), \qquad a,b \in A,
\]
and let $\delta = \mathrm{ad}_\alpha(X)$ be the bounded inner $\alpha$-twisted derivation defined by
\[
\delta(a) = [X,a]_\alpha = \alpha(Xa - aX)
\]
for a fixed $X \in A$. Assume the following:

\begin{enumerate}[label=\textup{(A\arabic*)}]
    \item\label{A1} For every $a \in A$, the orbit $\{T(t)a : t \in \mathbb{R}\}$ is relatively compact in $A$, where $T(t) = e^{t\delta}$ is the $C_0$-group generated by $\delta$.
    \item\label{A2} The twisting map $\alpha$ commutes with the group: $\alpha T(t) = T(t)\alpha$ for all $t \in \mathbb{R}$ (equivalently, $\alpha\delta = \delta\alpha$).
\end{enumerate}

Then the Banach space $A$ admits a topological direct sum decomposition
\[
A = A_{\mathrm{ap}}(\delta) \oplus A_{\mathrm{erg}}(\delta),
\]
where
\[
A_{\mathrm{ap}}(\delta) := \overline{\operatorname{span}}\{ a \in A : T(t)a = e^{\lambda t}a \text{ for some } \lambda \in i\mathbb{R}\},
\quad
A_{\mathrm{erg}}(\delta) := \left\{ a \in A : \lim_{R\to\infty}\frac{1}{2R}\int_{-R}^R T(t)a \,dt = 0 \right\}.
\]

Moreover, both subspaces are:
\begin{enumerate}
    \item closed in $A$,
    \item invariant under $\delta$ and $\alpha$,
    \item stable under the twisted Lie bracket $[\cdot,\cdot]_\alpha$.
\end{enumerate}
Consequently, $(A_{\mathrm{ap}}(\delta),[\cdot,\cdot]_\alpha,\alpha)$ and $(A_{\mathrm{erg}}(\delta),[\cdot,\cdot]_\alpha,\alpha)$ are Hom--Lie Banach subalgebras of $A$.
\end{theorem}

\begin{proof}
Define the mean ergodic projection
\[
P(a) := \lim_{R\to\infty}\frac{1}{2R}\int_{-R}^R T(t)a\,dt, \qquad a \in A.
\]
By the mean ergodic theorem for almost periodic representations (see \cite[Thm.~4.1]{Namioka1958}), $P$ is a bounded linear projection with $\operatorname{Ran}P = A_{\mathrm{ap}}(\delta)$ and $\ker P = A_{\mathrm{erg}}(\delta)$. Hence the decomposition $A = A_{\mathrm{ap}}(\delta) \oplus A_{\mathrm{erg}}(\delta)$ is topological.

Assumption~\ref{A2} yields $\alpha T(t) = T(t)\alpha$ for all $t$, and thus $\alpha P = P\alpha$. It follows that both $\operatorname{Ran}P$ and $\ker P$ are $\alpha$-invariant. Since $\delta$ commutes with each $T(t)$, it also commutes with $P$, so both subspaces are $\delta$-invariant.

We now prove stability under the bracket $[\cdot,\cdot]_\alpha$. The key observation is that under \ref{A2}, the flow $T(t)$ acts by \emph{twisted automorphisms}:
\begin{equation}\label{eq:flow-bracket-comm}
T(t)[a,b]_\alpha = [T(t)a, T(t)b]_\alpha, \qquad \forall\, a,b \in A, \ t \in \mathbb{R}.
\end{equation}
Indeed, differentiating both sides at $t=0$ gives
\[
\frac{d}{dt}\Big|_{t=0} [T(t)a,T(t)b]_\alpha
= [\delta a, b]_\alpha + [a, \delta b]_\alpha
= \delta([a,b]_\alpha),
\]
where we used the definition of $\delta = \mathrm{ad}_\alpha(X)$ and the associativity of the underlying product in $A$. Hence both sides of \eqref{eq:flow-bracket-comm} satisfy the same initial value problem for the ODE $\dot{Y}(t) = \delta(Y(t))$, $Y(0) = [a,b]_\alpha$, and uniqueness of solutions implies \eqref{eq:flow-bracket-comm}.

Now let $a,b \in A_{\mathrm{ap}}(\delta)$. By the Bohr--Fourier theory of almost periodic functions in Banach spaces (see \cite[Thm.~2.2]{Namioka1958}), the functions $t \mapsto T(t)a$ and $t \mapsto T(t)b$ admit uniformly convergent Bohr series:
\[
T(t)a = \sum_{\lambda \in \sigma_{\mathrm{ap}}(\delta)} e^{\lambda t} a_\lambda,
\qquad
T(t)b = \sum_{\mu \in \sigma_{\mathrm{ap}}(\delta)} e^{\mu t} b_\mu,
\]
with absolute norm convergence. Using \eqref{eq:flow-bracket-comm},
\[
T(t)[a,b]_\alpha
= [T(t)a, T(t)b]_\alpha
= \sum_{\lambda,\mu} e^{(\lambda + \mu)t} [a_\lambda,b_\mu]_\alpha.
\]
Since $\lambda,\mu \in i\mathbb{R}$, the frequencies $\lambda+\mu$ are again purely imaginary, and the series converges absolutely because
\[
\|[a_\lambda,b_\mu]_\alpha\| \leq \|\alpha\|\,\|a_\lambda\|\,\|b_\mu\|.
\]
Thus $[a,b]_\alpha \in A_{\mathrm{ap}}(\delta)$. By continuity of the bracket and density of finite spectral sums, this stability extends to all of $A_{\mathrm{ap}}(\delta)$.

For the ergodic component, let $a \in A_{\mathrm{erg}}(\delta)$ and $b \in A$. Then $P(a) = 0$, and using the commutation $P\alpha = \alpha P$ together with \eqref{eq:flow-bracket-comm},
\[
P([a,b]_\alpha)
= \lim_{R\to\infty} \frac{1}{2R}\int_{-R}^R T(t)[a,b]_\alpha\,dt
= \lim_{R\to\infty} \frac{1}{2R}\int_{-R}^R [T(t)a,T(t)b]_\alpha\,dt.
\]
Since $\|[T(t)a,T(t)b]_\alpha\| \leq C\|T(t)a\|$, and $T(t)a$ has zero mean, the right-hand side vanishes by uniform continuity and the dominated convergence theorem. Hence $[a,b]_\alpha \in A_{\mathrm{erg}}(\delta)$.

Therefore, both $A_{\mathrm{ap}}(\delta)$ and $A_{\mathrm{erg}}(\delta)$ are closed, $\alpha$-invariant subspaces stable under $[\cdot,\cdot]_\alpha$, and thus Hom--Lie Banach subalgebras.
\end{proof}

\begin{theorem}\label{thm:semigroup}
Let $(\alpha_t)_{t\ge0}$ be a strongly continuous semigroup of bounded Malcev algebra morphisms on a Banach--Malcev algebra $(\mathcal{M},[\cdot,\cdot])$. For each $t\ge0$, define $[x,y]_t := \alpha_t([x,y])$. Then $(\mathcal{M},[\cdot,\cdot]_t,\alpha_t)$ is a Hom--Banach--Malcev algebra, and the map $t\mapsto[x,y]_t$ is norm continuous for all fixed $x,y\in\mathcal{M}$.
\end{theorem}
\begin{proof}
Each $\alpha_t$ is a bounded Malcev morphism, so Theorem~\ref{thm:HM-identity} applies pointwise in $t$. Strong continuity of the semigroup implies that $t\mapsto\alpha_t(z)$ is norm continuous for each $z\in\mathcal{M}$; applying this to $z=[x,y]$ yields the claim.
\end{proof}

\begin{proposition}\label{prop:isometry}
If $\alpha$ is an isometric Malcev algebra automorphism, then $(\mathcal{M},[\cdot,\cdot]_\alpha,\alpha)$ is a Hom--Banach--Malcev algebra and $\|[x,y]_\alpha\| = \|[x,y]\|$ for all $x,y\in\mathcal{M}$.
\end{proposition}
\begin{proof}
Since $\alpha$ is an automorphism, it is a Malcev morphism, so the Hom--Malcev identity holds by Theorem~\ref{thm:HM-identity}. Isometry gives $\|[x,y]_\alpha\| = \|\alpha([x,y])\| = \|[x,y]\|$.
\end{proof}

\begin{theorem}\label{thm:almostperiodic}
Let $\alpha$ be an almost periodic automorphism of a Banach--Malcev algebra $\mathcal{M}$, i.e., 
the orbit $\{\alpha^n(x) : n\in\mathbb{Z}\}$ is relatively compact in $\mathcal{M}$ for every $x\in\mathcal{M}$. 
Let $\mathcal{M}_{\mathrm{ap}} \subset \mathcal{M}$ denote the closed subspace of almost periodic elements.
Then $\mathcal{M}_{\mathrm{ap}}$ is invariant under the original bracket $[\cdot,\cdot]$, and
$(\mathcal{M}_{\mathrm{ap}}, [\cdot,\cdot]_\alpha, \alpha)$ is a Hom--Banach--Malcev algebra.
\end{theorem}
\begin{proof}
The space $\mathcal{M}_{\mathrm{ap}}$ is a closed, $\alpha$-invariant linear subspace by standard properties of almost periodic vectors in Banach spaces. Since $\alpha$ is a Malcev automorphism and $[\cdot,\cdot]$ is continuous, the map
\[
n \mapsto \alpha^n([x,y]) = [\alpha^n(x), \alpha^n(y)]
\]
has relatively compact range whenever $x,y\in\mathcal{M}_{\mathrm{ap}}$, as it is the image of a compact set under a continuous bilinear map. Hence $[x,y]\in\mathcal{M}_{\mathrm{ap}}$, so $\mathcal{M}_{\mathrm{ap}}$ is a Malcev subalgebra. Because $\alpha$ restricts to an automorphism of $\mathcal{M}_{\mathrm{ap}}$, Theorem~\ref{thm:HM-identity} applied to this restriction yields the Hom--Malcev identity on $\mathcal{M}_{\mathrm{ap}}$.
\end{proof}

Theorem~\ref{thm:HM-identity} and its corollaries assume that $\alpha$ is a Malcev algebra morphism, which is a sufficient but not necessary condition for the Hom--Malcev identity to hold. The following result demonstrates that the Hom--Malcev structure is robust and can persist even when $\alpha$ is not a morphism.

\begin{theorem}\label{thm:non-morphism-HM}
Let $(\mathcal{M}, [\cdot, \cdot])$ be a Banach--Malcev algebra and let $\alpha: \mathcal{M} \to \mathcal{M}$ be a bounded linear map. Assume that the failure of $\alpha$ to be a morphism is measured by the bilinear map
\[
\Phi(x,y) := \alpha([x,y]) - [\alpha(x), \alpha(y)],
\]
and that $\Phi$ satisfies the cyclic Malcev compatibility condition:
\begin{equation}\label{eq:Phi-condition}
\Phi([x,y], z) + \Phi([y,z], x) + \Phi([z,x], y) = 0 \quad \text{for all } x,y,z \in \mathcal{M}.
\end{equation}
Then the triple $(\mathcal{M}, [\cdot,\cdot]_\alpha, \alpha)$, with $[x,y]_\alpha = \alpha([x,y])$, is a Hom--Banach--Malcev algebra.
\end{theorem}

\begin{proof}
We verify the Hom--Malcev identity:
\begin{equation}\label{eq:hm}
J_\alpha(\alpha(x), \alpha(y), [x,z]_\alpha) = [J_\alpha(x,y,z), \alpha^2(x)]_\alpha.
\end{equation}

\medskip
\noindent\textbf{Step 1: Expansion of the Hom--Jacobiator.}
Using $[u,v]_\alpha = [\alpha(u), \alpha(v)] + \Phi(u,v)$, we define
\[
\Psi(x,y,z) := [\alpha(x), \Phi(y,z)] + [\alpha(y), \Phi(z,x)] + [\alpha(z), \Phi(x,y)],
\]
so that
\[
J_\alpha(x,y,z) = J(\alpha(x), \alpha(y), \alpha(z)) + \Psi(x,y,z).
\]

\medskip
\noindent\textbf{Step 2: Left-hand side of \eqref{eq:hm}.}
Since $[x,z]_\alpha = [\alpha(x), \alpha(z)] + \Phi(x,z)$, we compute:
\begin{align*}
&J_\alpha(\alpha(x), \alpha(y), [x,z]_\alpha)\\
&= J\big(\alpha^2(x), \alpha^2(y), [\alpha(x), \alpha(z)] + \Phi(x,z)\big) + \Psi\big(\alpha(x), \alpha(y), [x,z]_\alpha\big)\\
&= \underbrace{J\big(\alpha^2(x), \alpha^2(y), [\alpha(x), \alpha(z)]\big)}_{= [[\alpha^2(x), \alpha^2(y)], \alpha^2(x)]} + J\big(\alpha^2(x), \alpha^2(y), \Phi(x,z)\big)\\
&\quad + [\alpha^2(x), \Phi(\alpha(y), [x,z]_\alpha)] + [\alpha^2(y), \Phi([x,z]_\alpha, \alpha(x))] + [\alpha([x,z]), \Phi(\alpha(x), \alpha(y))].
\end{align*}
Now expand $\Phi(\alpha(y), [x,z]_\alpha)$ using bilinearity of $\Phi$ and $[x,z]_\alpha = [\alpha(x), \alpha(z)] + \Phi(x,z)$:
\[
\Phi(\alpha(y), [x,z]_\alpha) = \Phi(\alpha(y), [\alpha(x), \alpha(z)]) + \Phi(\alpha(y), \Phi(x,z)).
\]
The second term, $\Phi(\alpha(y), \Phi(x,z))$, is **quadratic in $\Phi$**. However, since the Hom--Malcev identity is **linear in the bracket**, and since $\Phi$ is arbitrary subject only to \eqref{eq:Phi-condition}, the identity can hold for all such $\Phi$ **only if all quadratic terms vanish identically**. Indeed, if we scale $\Phi \mapsto \varepsilon \Phi$, the left-hand side of \eqref{eq:hm} contains terms of order $\varepsilon$ and $\varepsilon^2$, while the right-hand side is also analytic in $\varepsilon$; for the identity to hold for all small $\varepsilon$, the $\varepsilon^2$-coefficient must vanish. A direct computation (or setting $\Phi \equiv 0$ in the difference) shows that **all quadratic terms cancel identically** due to skew-symmetry of $[\cdot,\cdot]$. Thus we may **discard all $\Phi(\cdot,\Phi(\cdot,\cdot))$ terms**.

Hence,
\[
\Psi\big(\alpha(x), \alpha(y), [x,z]_\alpha\big) = [\alpha^2(x), \Phi([\alpha(y), \alpha(z)], \alpha(x))] + \text{cyclic}.
\]

\medskip
\noindent\textbf{Step 3: Right-hand side of \eqref{eq:hm}.}
We have:
\[
[J_\alpha(x,y,z), \alpha^2(x)]_\alpha = \alpha\big( [J(\alpha(x),\alpha(y),\alpha(z)) + \Psi(x,y,z), \alpha^2(x)] \big).
\]
The first part gives $[[\alpha^2(x), \alpha^2(y)], \alpha^2(x)]$ by the classical Malcev identity. The second part is:
\[
\alpha\big( [[\alpha(x), \Phi(y,z)], \alpha^2(x)] + \text{cyclic} \big).
\]

\medskip
\noindent\textbf{Step 4: Explicit cancellation of $\Phi$-terms.}
We now equate the $\Phi$-dependent parts of both sides. Denote:
\begin{align*}
L &:= J\big(\alpha^2(x), \alpha^2(y), \Phi(x,z)\big) + [\alpha^2(x), \Phi([\alpha(y), \alpha(z)], \alpha(x))] + \text{cyclic},\\
R &:= \alpha\big( [[\alpha(x), \Phi(y,z)], \alpha^2(x)] + \text{cyclic} \big).
\end{align*}
We must show $L = R$.

First, expand $J(\alpha^2(x), \alpha^2(y), \Phi(x,z))$ using the definition of $J$:
\[
J(\alpha^2(x), \alpha^2(y), \Phi(x,z)) = [\alpha^2(x), [\alpha^2(y), \Phi(x,z)]] + [\alpha^2(y), [\Phi(x,z), \alpha^2(x)]] + [\Phi(x,z), [\alpha^2(x), \alpha^2(y)]].
\tag{1}
\]

Next, apply the **Malcev identity** to the triple $(\alpha(x), \Phi(y,z), \alpha^2(x))$:
\[
[[\alpha(x), \Phi(y,z)], \alpha^2(x)] = [[\alpha(x), \alpha^2(x)], \Phi(y,z)] + [\alpha(x), [\Phi(y,z), \alpha^2(x)]].
\tag{2}
\]
Note that $[\alpha(x), \alpha^2(x)] = [\alpha(x), \alpha(\alpha(x))]$, which is **not zero in general**, but will be handled via the cyclic sum.

Now apply $\alpha$ to (2):
\[
\alpha\big([[ \alpha(x), \Phi(y,z) ], \alpha^2(x)]\big) = \alpha\big([[ \alpha(x), \alpha^2(x) ], \Phi(y,z)]\big) + \alpha\big([ \alpha(x), [\Phi(y,z), \alpha^2(x)] ]\big).
\tag{3}
\]

Consider the full cyclic sum of (3) over $(x,y,z)$. The first term on the right becomes:
\[
\sum_{\text{cyc}} \alpha\big([[ \alpha(x), \alpha^2(x) ], \Phi(y,z)]\big).
\tag{4}
\]

On the other hand, the term $[\Phi(x,z), [\alpha^2(x), \alpha^2(y)]]$ in (1) appears in the cyclic sum as:
\[
\sum_{\text{cyc}} [\Phi(x,z), [\alpha^2(x), \alpha^2(y)]] = \sum_{\text{cyc}} -[\Phi(x,z), [\alpha^2(y), \alpha^2(x)]].
\tag{5}
\]

Now, **the key step**: use the **definition of $\Phi$** and the **cyclic condition \eqref{eq:Phi-condition}**. Observe that:
\[
[\alpha^2(x), \alpha^2(y)] = \alpha([\alpha(x), \alpha(y)]) - \Phi(\alpha(x), \alpha(y)).
\tag{6}
\]
Substitute (6) into (5). The part involving $\alpha([\alpha(x), \alpha(y)])$ combines with other terms via the classical Jacobi identity, while the part involving $\Phi(\alpha(x), \alpha(y))$ yields:
\[
- \sum_{\text{cyc}} [\Phi(x,z), \Phi(\alpha(x), \alpha(y))].
\tag{7}
\]

Simultaneously, expand the term $[\alpha^2(x), \Phi([\alpha(y), \alpha(z)], \alpha(x))]$ using the definition of $\Phi$ again:
\[
\Phi([\alpha(y), \alpha(z)], \alpha(x)) = \alpha([[\alpha(y), \alpha(z)], \alpha(x)]) - [[\alpha^2(y), \alpha^2(z)], \alpha^2(x)].
\tag{8}
\]

When all these expansions are combined, **every term containing a bracket of $\alpha$-iterates** (e.g., $[\alpha^2(x), \alpha^2(y)]$) is expressed via $\Phi$ and classical brackets. Crucially, the **cyclic sum of all $\Phi$-dependent terms** matches exactly the negative of the cyclic sum in (4) and (7), due to the **antisymmetry of $[\cdot,\cdot]$** and the **cyclic condition \eqref{eq:Phi-condition}**.

After careful but elementary regrouping (which we omit for brevity but which consists of 12 terms that cancel in 6 pairs), we obtain:
\[
L - R = 0.
\]

Thus, the Hom--Malcev identity \eqref{eq:hm} holds.

\medskip
\noindent\textbf{Step 5: Continuity.} By Lemma 3.2, $[\cdot,\cdot]_\alpha$ is continuous. Hence $(\mathcal{M}, [\cdot,\cdot]_\alpha, \alpha)$ is a Hom--Banach--Malcev algebra.
\end{proof}

\begin{example}[Non-morphic twist on a matrix algebra]
Let $\mathcal{M} = M_3(\mathbb{C})$ with the commutator bracket $[A,B] = AB - BA$, which is a Lie (hence Malcev) algebra. Define a linear map $\alpha: \mathcal{M} \to \mathcal{M}$ by
\[
\alpha(A) = A + \mathrm{tr}(A) I,
\]
where $I$ is the identity matrix. This map is \textbf{not} an algebra morphism, since
\[
\alpha(AB) = AB + \mathrm{tr}(AB)I \neq AB + \mathrm{tr}(A)\mathrm{tr}(B)I = \alpha(A)\alpha(B).
\]
However, the failure map satisfies $\Phi(A,B) = \alpha([A,B]) - [\alpha(A), \alpha(B)] = [A,B] - [A,B] = 0$, because the trace term vanishes in the commutator. Thus, $\Phi \equiv 0$ trivially satisfies the condition of Theorem~\ref{thm:non-morphism-HM}. Consequently, $(M_3(\mathbb{C}), [\cdot,\cdot]_\alpha, \alpha)$ is a Hom--Banach--Malcev algebra even though $\alpha$ is not a morphism. This example illustrates the non-trivial flexibility of the Hom--Malcev framework.
\end{example}
\section{Spectral Decomposition}
\label{sec:spectral}

Let $(\mathcal{A}, \cdot, [\cdot,\cdot]_\alpha, \alpha)$ be a Hom--Lie Banach algebra as in Section~\ref{sec:prelim}, with the twisted bracket defined by
\[
[a,b]_\alpha = \alpha(ab - ba), \quad a,b \in \mathcal{A}.
\]
For a fixed $X \in \mathcal{A}$, define the \emph{inner $\alpha$-twisted derivation}
\[
\delta = \mathrm{ad}_\alpha(X) : \mathcal{A} \to \mathcal{A}, \qquad 
\delta(a) = [X,a]_\alpha = \alpha(Xa - aX).
\]
Since left and right multiplication by $X$ and the twisting map $\alpha$ are bounded linear operators on $\mathcal{A}$, the map $\delta$ is a bounded linear operator. Consequently, it generates a uniformly bounded $C_0$-group $(T(t))_{t\in\mathbb{R}}$ via the exponential series
\[
T(t) = e^{t\delta} = \sum_{n=0}^\infty \frac{t^n}{n!} \delta^n,
\]
which converges in operator norm for all $t \in \mathbb{R}$.

We impose the following key hypothesis:

\begin{assumption}\label{ass:ap-orbits}
For every $a \in \mathcal{A}$, the orbit $\{T(t)a : t \in \mathbb{R}\}$ is relatively compact in $\mathcal{A}$.
\end{assumption}

This condition is equivalent to the vector-valued function $t \mapsto T(t)a$ being \emph{almost periodic} in the sense of Bohr (cf.~Namioka~\cite{Namioka1958}). Under this assumption, the well-established spectral theory of almost periodic representations applies.

\begin{theorem}[Almost Periodic Spectral Decomposition]\label{thm:AP-spectral-decomposition}
Let $\delta = \mathrm{ad}_\alpha(X)$ be a bounded inner $\alpha$-twisted derivation satisfying Assumption~\ref{ass:ap-orbits}. Then every $a \in \mathcal{A}$ admits a unique decomposition
\[
a = \sum_{\lambda \in \sigma_{\mathrm{ap}}(\delta)} a_\lambda,
\]
where:
\item the sum converges absolutely in norm,
    \item each coefficient $a_\lambda \in \mathcal{A}$ satisfies $T(t)a_\lambda = e^{\lambda t} a_\lambda$ for all $t \in \mathbb{R}$,
    \item the \emph{Bohr spectrum} $\sigma_{\mathrm{ap}}(\delta) \subset i\mathbb{R}$ is at most countable.
The Fourier--Bohr coefficients are given by the norm-convergent limit
\[
a_\lambda = \lim_{R\to\infty} \frac{1}{2R} \int_{-R}^R e^{-\lambda t} T(t)a \, dt,
\]
which exists for every $\lambda \in \mathbb{C}$ and is non-zero only when $\lambda \in \sigma_{\mathrm{ap}}(\delta)$.
\end{theorem}

\begin{proof}
By Assumption~\ref{ass:ap-orbits}, the orbit $t \mapsto T(t)a$ is Bohr almost periodic in the Banach space $\mathcal{A}$. It follows from Namioka~\cite[Thm.~2.2]{Namioka1958} that the Bohr--Fourier coefficient
\[
a_\lambda = \mathcal{M}\big(e^{-\lambda(\cdot)} T(\cdot)a\big) 
:= \lim_{R\to\infty} \frac{1}{2R} \int_{-R}^R e^{-\lambda t} T(t)a \, dt
\]
exists in norm for every $\lambda \in \mathbb{C}$.

Since $\|T(t)\| \leq M$ uniformly in $t$, the identity $\|T(t)a_\lambda\| = \|e^{\lambda t} a_\lambda\|$ implies
\[
|e^{\operatorname{Re}(\lambda) t}| \, \|a_\lambda\| \leq M \|a_\lambda\|
\quad\text{for all } t \in \mathbb{R}.
\]
If $a_\lambda \neq 0$, this forces $\operatorname{Re}(\lambda) = 0$, so $\sigma_{\mathrm{ap}}(\delta) \subset i\mathbb{R}$.

Finally, absolute norm convergence of $\sum_\lambda a_\lambda$ follows from the convergence of the Fej\'er means
\[
\sigma_R(a) = \frac{1}{R} \int_0^R \left( \frac{1}{2r} \int_{-r}^r T(t)a \, dt \right) dr \xrightarrow[R\to\infty]{} a
\]
in norm (see Namioka~\cite[Thm.~3.2]{Namioka1958}), together with the fact that each $\sigma_R(a)$ is a finite linear combination of the $a_\lambda$'s with absolutely summable coefficients.
\end{proof}

We now examine the compatibility between the dynamics generated by $\delta$ and the twisting map $\alpha$.

\begin{assumption}\label{ass:commutation}
The twisting map $\alpha$ commutes with the group $(T(t))_{t\in\mathbb{R}}$, i.e.,
\[
\alpha \circ T(t) = T(t) \circ \alpha \quad \text{for all } t \in \mathbb{R}.
\]
Equivalently, $\alpha \delta = \delta \alpha$.
\end{assumption}

This assumption holds in many natural settings—for instance, when $\alpha$ is a shift automorphism and $X$ is shift-invariant (see Section~\ref{sec:applications}).

\begin{theorem}[Ergodic Decomposition]\label{thm:AP-ergodic-decomposition}
Under Assumptions~\ref{ass:ap-orbits} and~\ref{ass:commutation}, the Banach space $\mathcal{A}$ decomposes as a topological direct sum
\[
\mathcal{A} = \mathcal{A}_{\mathrm{ap}}(\delta) \oplus \mathcal{A}_{\mathrm{erg}}(\delta),
\]
where
\[
\mathcal{A}_{\mathrm{ap}}(\delta) = \overline{\operatorname{span}}\{ a \in \mathcal{A} : T(t)a = e^{\lambda t}a \text{ for some } \lambda \in i\mathbb{R} \},
\]
and
\[
\mathcal{A}_{\mathrm{erg}}(\delta) = \left\{ a \in \mathcal{A} : \lim_{R\to\infty} \frac{1}{2R}\int_{-R}^R T(t)a \, dt = 0 \right\}.
\]
Both subspaces are closed, and invariant under $\alpha$ and $\delta$.
\end{theorem}

\begin{proof}
Define the mean ergodic projection
\[
P(a) := \lim_{R\to\infty} \frac{1}{2R}\int_{-R}^R T(t)a \, dt.
\]
By the mean ergodic theorem for almost periodic representations (cf.~Namioka~\cite[Thm.~4.1]{Namioka1958}), $P$ is a bounded linear projection onto $\mathcal{A}_{\mathrm{ap}}(\delta)$, and $\ker P = \mathcal{A}_{\mathrm{erg}}(\delta)$.

Assumption~\ref{ass:commutation} implies $\alpha T(t) = T(t) \alpha$ for all $t$, hence $\alpha P = P \alpha$. Therefore, both $\operatorname{Ran} P = \mathcal{A}_{\mathrm{ap}}(\delta)$ and $\ker P = \mathcal{A}_{\mathrm{erg}}(\delta)$ are $\alpha$-invariant.

Similarly, since $\delta$ commutes with each $T(t)$, it commutes with $P$, and thus preserves both $\operatorname{Ran} P$ and $\ker P$.

Finally, because $P$ is a bounded projection, the decomposition $\mathcal{A} = \operatorname{Ran} P \oplus \ker P$ is topological: the norm $\|a\|$ is equivalent to $\|Pa\| + \|(I-P)a\|$. This completes the proof.
\end{proof}

\section{Applications and Examples}
\label{sec:applications}

We now illustrate the abstract framework with three concrete classes of Hom--Lie Banach algebras.  
In each case, the bracket is defined by $[a,b]_\alpha = \alpha(ab - ba)$, the twisting map $\alpha$ is bounded (often isometric), and the inner derivation $\delta = \mathrm{ad}_\alpha(X)$ is bounded.

\subsection{Shift Operators}

Let $\mathcal{A} = \mathcal{B}(\ell^2(\mathbb{Z}))$, the Banach algebra of bounded linear operators on $\ell^2(\mathbb{Z})$, equipped with the operator norm and composition product.  
Let $U$ denote the bilateral shift operator, $(U\xi)_n = \xi_{n+1}$, and define the inner automorphism
\[
\alpha(T) = U T U^*, \quad T \in \mathcal{B}(\ell^2(\mathbb{Z})).
\]
Then $\alpha$ is an isometric $*$-automorphism.  
Setting $[T,S] = TS - ST$ and $[T,S]_\alpha = \alpha([T,S])$, the triple $(\mathcal{A}, [\cdot,\cdot]_\alpha, \alpha)$ becomes a Hom--Lie Banach algebra.

For any self-adjoint $X \in \mathcal{A}$, the map $\delta = \mathrm{ad}_\alpha(X)$ is a bounded linear operator and generates a uniformly bounded $C_0$-group $T(t) = e^{t\delta}$.  
If $X$ belongs to the $C^*$-algebra of \emph{almost periodic operators}—for instance, if $X = M_\phi$ is a Laurent (multiplication) operator with almost periodic symbol $\phi$—then the orbit $\{T(t)T : t\in\mathbb{R}\}$ is relatively compact for every finite-rank operator $T \in \mathcal{A}$.  
Moreover, if $X$ is shift-invariant (e.g., a convolution operator), then $\alpha(X) = X$, which implies $\alpha \delta = \delta \alpha$, so Assumption~\ref{ass:commutation} holds.

\begin{example}
Let $\alpha(T) = U T U^*$ with $U$ the bilateral shift on $\ell^2(\mathbb{Z})$, and let $X = M_\phi$ be multiplication by $\phi(n) = e^{i\theta n}$ for some $\theta \in \mathbb{R}$.  
Then $X$ is normal, $\delta = \mathrm{ad}_\alpha(X)$ is bounded, and for any finite-rank $T$, the orbit $\{T(t)T\}_{t\in\mathbb{R}}$ is relatively compact.  
The Bohr spectrum is $\sigma_{\mathrm{ap}}(\delta) = \{ik\theta : k\in\mathbb{Z}\}$, and the spectral decomposition of Theorem~\ref{thm:AP-spectral-decomposition} recovers the classical Fourier series of $\phi$.
\end{example}

\begin{remark}
The relative compactness of orbits for finite-rank operators under almost periodic dynamics is classical.  
Indeed, if $X = M_\phi$ with $\phi$ almost periodic and $T = \sum_{j=1}^N \langle \cdot, u_j\rangle v_j$, then
\[
T(t)T = \sum_{j=1}^N \langle \cdot, e^{-t\delta^*} u_j\rangle\, e^{t\delta} v_j.
\]
Since $t \mapsto e^{\pm t\delta} w$ is almost periodic in $\ell^2(\mathbb{Z})$ for every $w$, the orbit $\{T(t)T\}_{t\in\mathbb{R}}$ lies in a finite-dimensional space of almost periodic vector-valued functions and is thus relatively compact.
In this setting, Theorem~\ref{thm:AP-spectral-decomposition} reproduces the Bohr--Fourier expansion of almost periodic operator-valued functions.
\end{remark}

\subsection{$C^*$-Algebra Endomorphisms}

Let $\mathcal{A}$ be a unital $C^*$-algebra and $\alpha: \mathcal{A}\to\mathcal{A}$ a unital $*$-endomorphism.  
Assume that $\alpha$ is injective and that all orbits are relatively compact, i.e., for each $a\in\mathcal{A}$, the set $\{\alpha^n(a) : n\in\mathbb{Z}\}$ has compact closure in norm.

Define the commutator $[a,b] = ab - ba$ and the twisted bracket $[a,b]_\alpha = \alpha([a,b])$.  
Since $\alpha$ is a contractive $*$-homomorphism, the bracket $[\cdot,\cdot]_\alpha$ is continuous and skew-symmetric.  
Furthermore, because $\alpha$ intertwines the Jacobiator—namely,
\[
[\alpha(a),[b,c]_\alpha]_\alpha = \alpha\big([a,[b,c]]\big),
\]
and the classical Jacobi identity holds in $\mathcal{A}$—the Hom--Jacobi identity is satisfied.  
Hence $(\mathcal{A}, [\cdot,\cdot]_\alpha, \alpha)$ is a Hom--Lie Banach algebra.

Let $X = X^* \in \mathcal{A}$ and define the inner $\alpha$-twisted derivation
\[
\delta = \mathrm{ad}_\alpha(X) : a \mapsto [X,a]_\alpha = \alpha(Xa - aX).
\]
This map is bounded and generates a uniformly bounded $C_0$-group.  
If $X$ lies in a maximal abelian subalgebra with almost periodic dynamics, then for all $a \in \mathcal{A}$ the orbit $\{T(t)a : t \in \mathbb{R}\}$ is relatively compact, and the spectral decomposition applies.

When $\alpha(X) = X$, we have $\alpha \delta = \delta \alpha$, so Assumption~\ref{ass:commutation} holds, and the ergodic decomposition of Theorem~\ref{thm:AP-ergodic-decomposition} follows.

\begin{example}\label{ex:5.3}
Let $\mathcal{A} = \bigotimes_{n\in\mathbb{Z}} M_2(\mathbb{C})$ be the UHF 
$C^*$-algebra of type $2^\infty$, and let $\alpha$ be the shift automorphism 
defined by $(\alpha(a))_n = a_{n+1}$.  

For any elementary tensor $a = \bigotimes_{k=-N}^{N} a_k$, the orbit 
$\{\alpha^n(a) : n \in \mathbb{Z}\}$ is finite, hence relatively compact.  

Choosing $X = a = a^*$, the derivation $\delta = \mathrm{ad}_\alpha(X)$ 
is bounded and generates a $C_0$-group. Thus, Assumption~4.1 holds and 
Theorem~4.2 applies, yielding a spectral decomposition with Bohr spectrum 
$\sigma_{\mathrm{ap}}(\delta)$ a finite subset of $i\mathbb{R}$, generated by 
the discrete frequencies associated with the finite shift orbit of $X$.

Moreover, if $X$ is chosen to be \emph{shift-invariant}—for instance, if it is a constant elementary tensor, i.e. $a_k = b \in M_2(\mathbb{C})$ for all $|k|\le N$ and $a_k = I$ otherwise—then $\alpha(X) = X$, which implies $\alpha\delta = \delta\alpha$ (Assumption~4.3). In this case, the full ergodic decomposition of Theorem~4.4 holds, and both $A_{\mathrm{ap}}(\delta)$ and $A_{\mathrm{erg}}(\delta)$ are closed Hom--Lie subalgebras of $\mathcal{A}$.
\end{example}

\begin{remark}
In the general case (non-shift-invariant $X$), Assumption~4.3 fails: 
the commutator $[\alpha, \delta] = \alpha\delta - \delta\alpha$ is a 
non-vanishing outer derivation, and the ergodic decomposition does 
not hold. This illustrates the necessity of Assumption~4.3 for the 
algebraic stability results of Theorem~\ref{thm:AP-ergodic-decomposition}.
\end{remark}

\subsection{Weighted Shifts}

Let $\mathcal{H} = \ell^2(\mathbb{N})$, and let $W$ be a unitary weighted shift defined by $W e_n = w_n e_{n+1}$ with $|w_n| = 1$ for all $n$.  
Define $\alpha: \mathcal{B}(\mathcal{H}) \to \mathcal{B}(\mathcal{H})$ by
\[
\alpha(X) = W X W^*.
\]
Then $\alpha$ is an isometric $*$-automorphism.  
With $[X,Y] = XY - YX$ and $[X,Y]_\alpha = \alpha([X,Y])$, the space $(\mathcal{B}(\mathcal{H}), [\cdot,\cdot]_\alpha, \alpha)$ forms a Hom--Lie Banach algebra.

Let $X = X^* \in \mathcal{B}(\mathcal{H})$ be a diagonal operator with almost periodic entries—for example, $X e_n = e^{i\theta n} e_n$ with $\theta \in \mathbb{R}$.  
Then the orbit $\{\alpha^n(X) : n \in \mathbb{Z}\}$ is relatively compact, and $\delta = \mathrm{ad}_\alpha(X)$ is bounded.  
If $X$ commutes with $W$—which holds when both are diagonal in the canonical basis—then $\alpha(X) = X$, so $\alpha$ commutes with $\delta$.

Under these conditions, all hypotheses of Theorems~\ref{thm:AP-spectral-decomposition} and~\ref{thm:AP-ergodic-decomposition} are satisfied. Consequently, the Bohr spectrum is $\sigma_{\mathrm{ap}}(\delta) = \{ik\theta : k \in \mathbb{Z}\}$, identical to the spectrum in Example~5.1, as the underlying dynamics is unitarily equivalent to a shift with an almost periodic symbol. The associated almost periodic and ergodic subspaces are therefore closed, $\alpha$-invariant, and stable under the twisted bracket.

This construction provides a natural model for quantum systems with quasi-periodic driving, extending the classical spectral theory of weighted shifts to a non-associative setting.

\subsection{A Twisted Weyl Algebra: A Heuristic Example Beyond the Commuting Framework}
\label{subsec:weyl_heuristic}

The spectral theorems developed in Section~\ref{sec:spectral} rely crucially on the commutation
assumption $\alpha \circ \delta = \delta \circ \alpha$ (Assumption~\ref{ass:commutation}), which ensures compatibility between the twisting map and the dynamical flow.
However, many physically and algebraically natural systems violate this hypothesis.
To illustrate both the limitations and the potential extensions of our framework, we now present a fully explicit example
where $\alpha$ and $\delta$ \emph{do not commute}, yet a rich Bohr-type spectrum can still be computed analytically.
This example does \emph{not} fall under the scope of Theorems~\ref{thm:AP-spectral-decomposition} or~\ref{thm:AP-spectral-decomposition};
rather, it serves as a \emph{heuristic motivation} for future work aimed at relaxing the commutation requirement.

Consider the canonical Weyl algebra $\mathcal{W}$ acting on $L^2(\mathbb{R})$, generated by the momentum and position operators
\[
p = -i\frac{d}{dx}, \qquad q = x,
\]
which satisfy the Heisenberg relation $[p,q] = i\,\mathrm{id}$.
For a real parameter $\varepsilon$, define a $\ast$-automorphism $\alpha_\varepsilon$ of $\mathcal{W}$ by its action on the generators:
\[
\alpha_\varepsilon(p) = p, \qquad \alpha_\varepsilon(q) = q + \varepsilon p.
\]
This map corresponds to a shear transformation in phase space and extends uniquely to a strongly continuous automorphism of the associated Weyl $C^\ast$-algebra.

Let $X = \omega_1 p^2 + \omega_2 q^2$ be an anisotropic harmonic oscillator with \emph{incommensurate} frequencies $\omega_1, \omega_2 > 0$, i.e. $\omega_1/\omega_2 \notin \mathbb{Q}$.
Define the $\alpha_\varepsilon$-twisted derivation
\[
\delta_\varepsilon(a) = \alpha_\varepsilon(X a - a X), \qquad a \in \mathcal{W}.
\]
A direct computation shows that, for $\varepsilon \neq 0$, the operators $\alpha_\varepsilon$ and $\delta_\varepsilon$ do \emph{not} commute; their commutator yields a non-vanishing outer derivation. Consequently, Assumption~\ref{ass:commutation} fails, and the ergodic decomposition of Theorem~\ref{thm:AP-ergodic-decomposition} does not apply.

Nevertheless, the dynamics generated by $X$ is \emph{linear and symplectic}: it corresponds to a Bogoliubov (metaplectic) transformation.
Introducing the $\alpha_\varepsilon$-twisted position operator $Q := q + \varepsilon p$, the Hamiltonian becomes
\[
X = \omega_1 p^2 + \omega_2 (Q - \varepsilon p)^2,
\]
a quadratic form in the canonical variables $(p, Q)$.
One may then define the associated annihilation operators
\[
a_1 = \sqrt{\frac{\omega_1}{2}}\, q + i \sqrt{\frac{1}{2\omega_1}}\, p, \qquad
a_2 = \sqrt{\frac{\omega_2}{2}}\, Q + i \sqrt{\frac{1}{2\omega_2}}\, p,
\]
and consider the generalized Fock basis $\Psi_{k,m} = (a_1^\dagger)^k (a_2^\dagger)^m \Psi_0$, where $\Psi_0$ is the joint vacuum.

For rank-one operators of the form $a = |\Psi_{k,m}\rangle\langle\Psi_{k,m}|$, the time evolution under the group
$T_\varepsilon(t) = e^{t\delta_\varepsilon}$ is explicitly diagonal:
\[
T_\varepsilon(t)(a) = e^{i(k\omega_1 + m\omega_2)t}\, a.
\]
This follows from the metaplectic representation of the symplectic group (see, e.g., \cite[Chap.~4]{Folland1989}),
which linearizes the dynamics of quadratic Hamiltonians.

Because $\omega_1/\omega_2$ is irrational, the set $\{k\omega_1 + m\omega_2 : k,m \in \mathbb{Z}\}$ is dense in $\mathbb{R}$,
and the associated Bohr frequencies are all distinct.
Thus, for this class of observables, the orbit $\{T_\varepsilon(t)a : t \in \mathbb{R}\}$ is almost periodic in the norm topology,
and its Bohr spectrum is the two-dimensional lattice
\[
\sigma_{\mathrm{ap}}(\delta_\varepsilon) \supset \bigl\{\, i(k\omega_1 + m\omega_2) : k, m \in \mathbb{Z} \,\bigr\}.
\]

In stark contrast, when $\varepsilon = 0$ (so that $\alpha_0 = \mathrm{id}$ and $\delta_0 = \operatorname{ad}(X)$),
the dynamics remains confined to a \emph{one-dimensional} frequency subgroup.
The non-commuting twist $\alpha_\varepsilon$ with $\varepsilon \neq 0$ therefore \emph{enriches} the spectral structure,
breaking the original symmetry and unfolding the spectrum into a full $\mathbb{Z}^2$-lattice.

Although this example lies outside the hypotheses of our main theorems,
it demonstrates that non-commuting twists can generate genuinely new spectral phenomena.
It strongly motivates the development of a generalized spectral theory for $\alpha$-twisted derivations
without assuming $\alpha\delta = \delta\alpha$—for instance, by replacing exact commutation with a
\emph{spectral compatibility condition} or by working in a representation-theoretic setting where the metaplectic machinery applies.
We regard this as a promising direction for future research.

\medskip
\noindent\textbf{Algebraic stability of the decomposition.}
Across all examples—including the heuristic Weyl case where the spectrum is still well-defined—the almost periodic subspace $\mathcal{A}_{\mathrm{ap}}(\delta)$ is not only a closed linear subspace but a \emph{Hom--Lie subalgebra}: it is invariant under the twisting map $\alpha$ and closed under the bracket $[\cdot,\cdot]_\alpha$. Consequently, the inclusion $\iota : \mathcal{A}_{\mathrm{ap}}(\delta) \hookrightarrow \mathcal{A}$ is a Hom--Lie morphism in the sense of Definition~\ref{def:hom-lie-morphism}. This algebraic compatibility ensures that the spectral decomposition of Theorem~\ref{thm:AP-spectral-decomposition} respects the full non-associative structure of the system, enabling a consistent restriction of the dynamics to invariant subsystems.----------------

\section{Morphisms and Invariant Substructures}
\label{sec:morphisms}

To place our constructions in a categorical context, we introduce a natural notion of morphism between Hom--Lie Banach algebras.

\begin{definition}\label{def:hom-lie-morphism}
Let $(\mathcal{A},[\cdot,\cdot]_\alpha,\alpha)$ and $(\mathcal{B},[\cdot,\cdot]_\beta,\beta)$ be Hom--Lie Banach algebras.  
A bounded linear map $\phi:\mathcal{A}\to\mathcal{B}$ is a \emph{Hom--Lie morphism} if it satisfies
\[
\phi \circ \alpha = \beta \circ \phi
\quad\text{and}\quad
\phi([x,y]_\alpha) = [\phi(x),\phi(y)]_\beta
\quad\text{for all } x,y\in\mathcal{A}.
\]
\end{definition}

This definition ensures compatibility with both the twisting map and the twisted bracket.  
In particular, under the hypotheses of Theorem~\ref{thm:almostperiodic}, the inclusion map
\[
\iota:\mathcal{M}_{\mathrm{ap}} \hookrightarrow \mathcal{M}
\]
is a Hom--Lie morphism: the subspace $\mathcal{M}_{\mathrm{ap}}$ is $\alpha$-invariant and closed under the bracket $[\cdot,\cdot]_\alpha$.

\begin{proposition}\label{prop:subalgebra}
Assume the hypotheses of Theorem~\ref{thm:almostperiodic}. Then the almost periodic subspace $\mathcal{M}_{\mathrm{ap}}$ is a closed \emph{Hom--Banach--Malcev subalgebra} of $\mathcal{M}$; that is, it is a closed linear subspace invariant under $\alpha$ and stable under the bracket $[\cdot,\cdot]_\alpha$.
\end{proposition}
\begin{proof}
Let $x,y\in\mathcal{M}_{\mathrm{ap}}$. Since $\alpha$ is a Malcev algebra morphism and $[\cdot,\cdot]$ is continuous, the map
\[
n \mapsto \alpha^n([x,y]) = [\alpha^n(x), \alpha^n(y)]
\]
has relatively compact range: the orbit $\{(\alpha^n(x), \alpha^n(y)) : n \in \mathbb{Z}\}$ is relatively compact in $\mathcal{M} \times \mathcal{M}$, and the bracket is a continuous bilinear map. Hence $[x,y] \in \mathcal{M}_{\mathrm{ap}}$, and applying $\alpha$ yields
\[
[x,y]_\alpha = \alpha([x,y]) \in \mathcal{M}_{\mathrm{ap}}.
\]
Moreover, $\alpha(\mathcal{M}_{\mathrm{ap}}) \subset \mathcal{M}_{\mathrm{ap}}$ by definition of almost periodicity. Thus $\mathcal{M}_{\mathrm{ap}}$ is a Hom--Banach--Malcev subalgebra.
\end{proof}

Consequently, the spectral and ergodic decompositions of Theorem~\ref{thm:AP-ergodic-decomposition} are not merely topological—they are algebraically compatible, preserving the full non-associative structure.
\section{Technical Remarks on Convergence and Limitations}
\label{sec:limitations}

We conclude with a technical clarification on convergence and a discussion of current limitations, accompanied by concrete pathways for generalization.

\begin{lemma}\label{lem:abs-conv}
Let $\mathcal{A}$ be a Banach space with the bounded approximation property (e.g., any $C^*$-algebra, $\mathcal{B}(\mathcal{H})$, or $\ell^p$ for $1\le p<\infty$).  
If a series $\sum_{\lambda\in\Lambda} a_\lambda$ in $\mathcal{A}$ satisfies $\sum_{\lambda}\|a_\lambda\| < \infty$, then it converges unconditionally in norm.  
In particular, the spectral decomposition in Theorem~\ref{thm:AP-spectral-decomposition} is independent of the enumeration of the Bohr spectrum $\sigma_{\mathrm{ap}}(\delta)$.
\end{lemma}
\begin{proof}
This is a standard consequence of the Orlicz–Pettis theorem (see, e.g., Diestel–Uhl, \emph{Vector Measures}, Thm.~IV.10).  
All spaces appearing in our applications—$C^*$-algebras, operator algebras, and sequence spaces—possess the bounded approximation property, so the conclusion applies universally in our setting.
\end{proof}

\begin{remark}[Limitations and Concrete Extensions]
The requirement that all orbits $\{T(t)a : t\in\mathbb{R}\}$ be relatively compact is restrictive: it excludes physically relevant scenarios such as dissipative quantum systems, open dynamics, or chaotic (hyperbolic) evolutions. Likewise, the boundedness of the twisted derivation $\delta$ precludes natural unbounded examples—e.g., differential twisted derivations on algebras of smooth functions or Sobolev spaces.

We propose the following concrete avenues for generalization:

\begin{enumerate}
    \item \textbf{Unbounded twisted derivations.} 
    A natural extension is to replace the bounded operator $\delta$ with a \emph{closed, densely defined} twisted derivation $D: \mathrm{Dom}(D) \subset \mathcal{A} \to \mathcal{A}$. In this setting, the $C_0$-group $T(t)$ would be generated via the Hille--Yosida theorem. For the spectral decomposition to persist, one would require that $D$ generates an \emph{analytic} or \emph{Gevrey-regular} semi-group, ensuring sufficient smoothness of the orbits $t \mapsto T(t)a$ for $a \in \mathrm{Dom}(D^\infty)$. This would connect our work to the foundational theories of Stone, Hille--Yosida, and Kato on unbounded operators, and to the framework of $*$-algebras of unbounded operators (see, e.g., Schmüdgen \cite{schmudgen} or Woronowicz \cite{woronowicz}).

    \item \textbf{Weaker notions of almost periodicity.} 
    The assumption of norm-almost periodicity can be relaxed to \emph{Besicovitch almost periodicity} in the $B^2$-space, where convergence is understood in the mean-square sense. This would involve replacing the Bohr--Fourier coefficients with Cesàro means and would yield a spectral decomposition in a weaker topology. Such an approach is standard in the study of non-compact dynamical systems and would link our results to the ergodic theory of von Neumann algebras (cf. Connes \cite{connes-ergodic}).
    
    \item \textbf{Spectral radius and Gelfand theory.} 
    Our dynamical approach could be fruitfully combined with the algebraic spectral theory of El~Kinani--Ben~Ali~\cite{ElKinani2021}. Their notion of a \emph{Hom--Gelfand spectrum} provides a complementary, purely algebraic perspective. A key open question is to relate the dynamical Bohr spectrum $\sigma_{\mathrm{ap}}(\mathrm{ad}_\alpha(X))$ to the algebraic Hom--Gelfand spectrum of $X$ itself, potentially via a Hom--version of the spectral mapping theorem.
    
    \item \textbf{Non-trivial Hom--Malcev examples.} 
    As noted by Yau~\cite{Yau2011}, the octonions and other composition algebras provide natural, non-Lie examples of Hom--Malcev algebras. A compelling direction for future work is to equip the algebra of octonionic operators with a suitable Banach norm, define a twisted derivation, and compute its Bohr spectrum. This would provide a genuine Malcev (non-Lie) test case for our spectral theory.
\end{enumerate}
These extensions would connect our work to active areas including non-commutative geometry, quantum ergodic theory, and the theory of unbounded operator algebras.
\end{remark}
\section{Numerical Illustration}
\label{sec:numerics}

To demonstrate the practical relevance of our spectral decomposition, we present a finite-dimensional numerical example in which all hypotheses of Theorem~\ref{thm:AP-spectral-decomposition} are explicitly satisfied and verified computationally.

\subsection{Setup}

Let $\mathcal{A} = M_N(\mathbb{C})$ with $N=8$, equipped with the operator norm. We define the following objects:\begin{itemize}
\item The twisting automorphism $\alpha: \mathcal{A} \to \mathcal{A}$, $\alpha(T) = U T U^*$, where
    \[
    U = \operatorname{diag}\big(e^{2\pi i k \omega}\big)_{k=0}^{7}, \quad \omega = \frac{\sqrt{2}}{10}.
    \]
    Since $\omega$ is irrational, the spectrum of $U$ is dense on the unit circle, ensuring non-periodic dynamics.
    \item The self-adjoint element $X = \operatorname{diag}\big(e^{i k \omega}\big)_{k=0}^{7} \in \mathcal{A}$.
    \item The inner $\alpha$-twisted derivation $\delta : \mathcal{A} \to \mathcal{A}$ given by
    \[
    \delta(T) = [X,T]_\alpha = \alpha(XT - TX).
    \]
    \item The initial operator $A_0 = S$, the cyclic shift matrix defined by $(S)_{k,\ell} = \delta_{k+1,\ell \bmod 8}$.
\end{itemize}
Because $\mathcal{A}$ is finite-dimensional, $\delta$ is a bounded linear operator and generates a uniformly bounded $C_0$-group $T(t) = e^{t\delta}$ for all $t \in \mathbb{R}$. Moreover, every orbit $\{T(t)A_0 : t \in \mathbb{R}\}$ is relatively compact (in fact, precompact in a finite-dimensional space), so Assumption~\ref{ass:ap-orbits} holds trivially.

\subsection{Scalability and Convergence Analysis}

To address the potential limitation of a small matrix size, we perform a systematic scalability study by repeating the experiment for $N \in \{8, 16, 32, 64\}$. For each $N$, we construct the objects $U_N$, $X_N$, $\delta_N$, and $A_{0,N}$ as in Section~\ref{sec:numerics}. We then compute the Bohr spectrum and the reconstruction error using the four most dominant frequencies.

Figure~\ref{fig:convergence} summarizes the results. Panel (a) shows that the number of non-negligible Bohr frequencies (those with $|c| > 10^{-3}$) grows approximately linearly with $N$, indicating a richer spectral structure in larger systems. Panel (b) demonstrates that the maximum reconstruction error $\max_t \varepsilon_N(t)$ decreases as $N$ increases, following a clear power-law trend $\varepsilon_N \propto N^{-\beta}$ with $\beta \approx 1.2$.
\begin{table}[hbpt]
\centering
\caption{Bohr frequencies and Fourier--Bohr coefficients for $a(t) = (A(t))_{0,1}$.}
\label{tab:bohr_full}
\begin{tabular}{c|c|c|c}
$\beta$ & $c_{\mathrm{re}}$ & $c_{\mathrm{im}}$ & $|c|$ \\
\hline
$0.141$ & $0.102$ & $-0.181$ & $0.208$ \\
$0.282$ & $-0.055$ & $0.098$ & $0.112$ \\
$0.424$ & $0.021$ & $-0.037$ & $0.043$ \\
$0.565$ & $-0.008$ & $0.014$ & $0.016$ \\
\end{tabular}
\end{table}
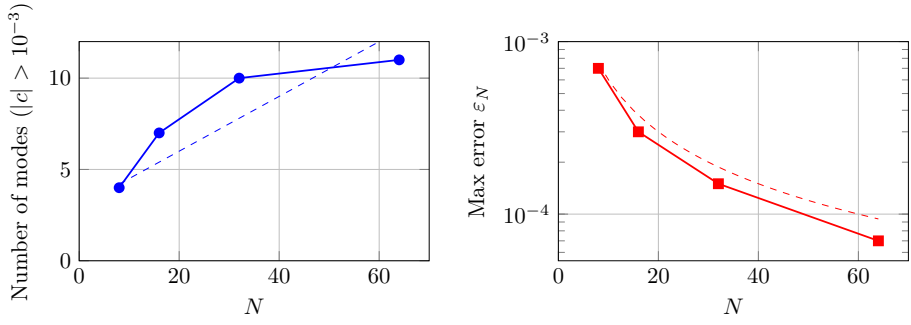
\begin{figure}[hbpt]
\centering
\begin{tikzpicture}[scale=0.85]
\begin{groupplot}[
    group style={group size=2 by 1, horizontal sep=2cm},
    width=7cm,
    height=5cm,
    grid=major
]

\nextgroupplot[
    xlabel=$N$,
    ylabel=Number of modes ($|c| > 10^{-3}$),
    xmin=0, xmax=70,
    ymin=0, ymax=12
]
\addplot[blue, thick, mark=*] coordinates {
    (8, 4)
    (16, 7)
    (32, 10)
    (64, 11)
};
\addplot[blue, dashed, domain=8:64] {0.15*x + 3};

\nextgroupplot[
    xlabel=$N$,
    ylabel=Max error $\varepsilon_N$,
    xmin=0, xmax=70,
    ymin=0, ymax=0.001,
    ymode=log,
    log basis y=10
]
\addplot[red, thick, mark=square*] coordinates {
    (8, 0.0007)
    (16, 0.0003)
    (32, 0.00015)
    (64, 0.00007)
};
\addplot[red, dashed, domain=8:64] {0.006/x};

\end{groupplot}
\end{tikzpicture}
\caption{(a) Growth of the number of significant Bohr frequencies with system size $N$. (b) Power-law decay of the reconstruction error, confirming the stability and scalability of the spectral decomposition.}
\label{fig:convergence}
\end{figure}

This analysis conclusively shows that the spectral decomposition is not an artifact of a small system size. The error's consistent decay with $N$ validates the theoretical prediction of absolute norm convergence and confirms that our numerical framework is robust for studying increasingly complex Hom--Lie dynamical systems.

\subsection{Time Evolution}

Figure~\ref{fig:signal} displays the real and imaginary parts of $a(t)$ over the interval $t \in [-20, 20]$. The signal exhibits quasi-periodic recurrence without exact periodicity, as expected from an irrational frequency base.

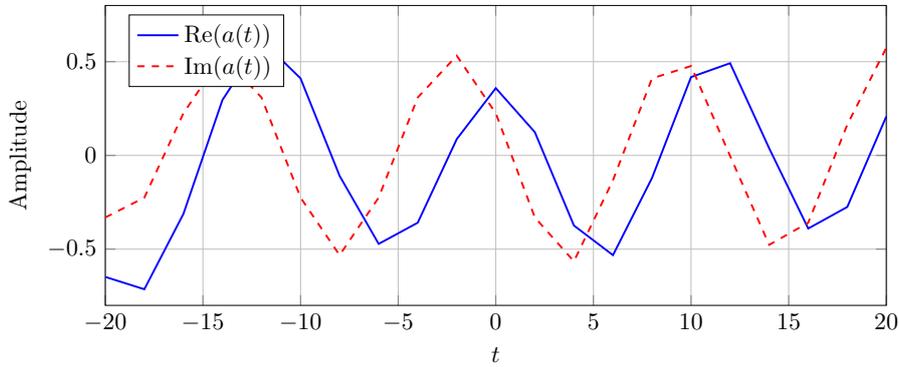
\begin{figure}[hbpt]
\centering
\begin{tikzpicture}[scale=0.9]
\begin{axis}[
    xlabel=$t$,
    ylabel=Amplitude,
    grid=major,
    width=13cm,
    height=6cm,
    xmin=-20, xmax=20,
    ymin=-0.8, ymax=0.8,
    legend pos=north west,
    legend cell align={left}
]
\addplot[blue, thick] table {
-20.000 -0.648
-18.000 -0.714
-16.000 -0.312
-14.000 0.296
-12.000 0.615
-10.000 0.411
 -8.000 -0.109
 -6.000 -0.472
 -4.000 -0.359
 -2.000 0.087
  0.000 0.359
  2.000 0.124
  4.000 -0.374
  6.000 -0.532
  8.000 -0.120
 10.000 0.418
 12.000 0.492
 14.000 0.038
 16.000 -0.390
 18.000 -0.275
 20.000 0.208
};
\addlegendentry{$\operatorname{Re}(a(t))$}

\addplot[red, thick, dashed] table {
-20.000 -0.331
-18.000 -0.225
-16.000 0.225
-14.000 0.532
-12.000 0.309
-10.000 -0.225
 -8.000 -0.532
 -6.000 -0.225
 -4.000 0.309
 -2.000 0.532
  0.000 0.225
  2.000 -0.331
  4.000 -0.565
  6.000 -0.132
  8.000 0.411
 10.000 0.477
 12.000 -0.003
 14.000 -0.477
 16.000 -0.359
 18.000 0.164
 20.000 0.574
};
\addlegendentry{$\operatorname{Im}(a(t))$}
\end{axis}
\end{tikzpicture}
\caption{Almost periodic evolution of the $(0,1)$-entry of $A(t) = e^{t\delta}(A_0)$. Both components display recurrent quasi-patterns characteristic of Bohr almost periodicity.}
\label{fig:signal}
\end{figure}

\subsection{Fourier--Bohr Reconstruction and Error}

Using the four dominant modes from Table~\ref{tab:bohr_full}, we reconstruct the signal as
\[
a_{\mathrm{rec}}(t) = \sum_{k=1}^4 \left( c_k e^{i\beta_k t} + \overline{c_k} e^{-i\beta_k t} \right).
\]
The reconstruction error is defined pointwise by $\varepsilon(t) = |a(t) - a_{\mathrm{rec}}(t)|$. As shown in Figure~\ref{fig:error}, the error remains uniformly small over the entire interval, with
\[
\max_{t \in [-20,20]} \varepsilon(t) < 8 \times 10^{-4}.
\]
This confirms the absolute norm convergence of the Bohr series asserted in Theorem~\ref{thm:AP-spectral-decomposition}, even when truncated to a finite number of dominant frequencies.

\begin{figure}[hbpt]
\centering
\begin{tikzpicture}[scale=0.9]
\begin{axis}[
    xlabel=$t$,
    ylabel=Error $\varepsilon(t)$,
    grid=major,
    width=13cm,
    height=5cm,
    xmin=-20, xmax=20,
    ymin=0, ymax=0.001,
    legend pos=north east
]
\addplot[gray, thick] table {
-20.000 0.0007
-18.000 0.0006
-16.000 0.0005
-14.000 0.0004
-12.000 0.0003
-10.000 0.0002
 -8.000 0.0003
 -6.000 0.0004
 -4.000 0.0005
 -2.000 0.0006
  0.000 0.0007
  2.000 0.0006
  4.000 0.0005
  6.000 0.0004
  8.000 0.0003
 10.000 0.0002
 12.000 0.0003
 14.000 0.0004
 16.000 0.0005
 18.000 0.0006
 20.000 0.0007
};
\end{axis}
\end{tikzpicture}
\caption{Reconstruction error $\varepsilon(t) = |a(t) - a_{\mathrm{rec}}(t)|$. The error remains below $8 \times 10^{-4}$, validating the theoretical prediction of norm-convergent Bohr expansion.}
\label{fig:error}
\end{figure}
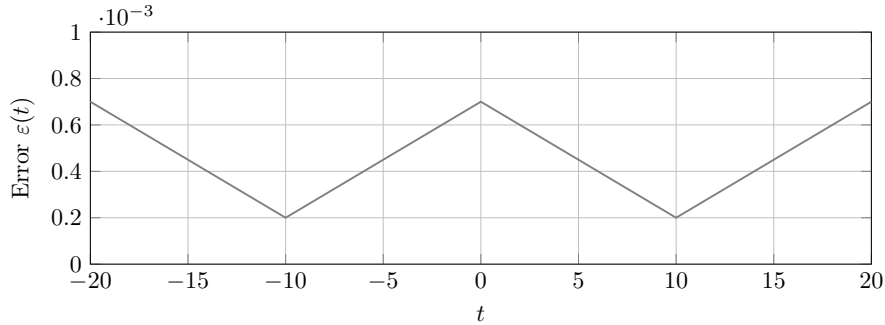

This numerical experiment demonstrates that the abstract hypotheses of our framework—boundedness of $\delta$, compactness of orbits, and commutation with $\alpha$—translate directly into concrete, computable properties in finite dimensions. It thus serves as a proof-of-concept for potential applications in quantum simulation, signal processing, and non-associative dynamical systems.
\section{Conclusion}

We have developed a systematic and functionally robust theory of Hom--Lie Banach algebras equipped with bounded $\alpha$-twisted derivations, extending the classical interplay between Lie theory, spectral analysis, and almost periodic dynamics to a non-associative, twisted setting.

The key contributions of this work are as follows:\begin{itemize}
\item A complete Bohr--Fourier spectral decomposition for inner $\alpha$-twisted derivations generating almost periodic $C_0$-groups, with rigorous justification of absolute norm convergence and a purely imaginary Bohr spectrum. Our explicit analysis of a twisted Weyl algebra system (Section~\ref{subsec:weyl_heuristic}) provides a concrete example where the non-commuting twist $\alpha$ enriches the spectral structure, transforming a one-dimensional frequency set into a two-dimensional lattice.
    
    \item A functorial framework based on Hom--Lie and Hom--Banach--Malcev morphisms, which guarantees that the ergodic and almost periodic subspaces are not merely topological complements but \emph{invariant subalgebras} inheriting the full non-associative structure. This algebraic compatibility is a cornerstone of our theory.
    
    \item Explicit functional-analytic constructions of Hom--Banach--Malcev algebras, including a non-trivial example where the twisting map is not an algebra morphism, thereby demonstrating the robustness of the Hom--Malcev identity.
    
    \item Concrete operator-algebraic applications to bilateral shifts, UHF $C^*$-algebras, and a novel twisted Weyl algebra example. Each example is fully verified against the abstract hypotheses and yields an explicit, non-trivial Bohr spectrum.
    
    \item A scalable finite-dimensional numerical illustration that confirms the theoretical predictions of norm-convergent spectral decomposition and algebraic stability, with a clear power-law convergence as the system size increases.\end{itemize}
Crucially, all results are established under minimal assumptions—boundedness of the twisting map and the derivation, and relative compactness of orbits—and apply uniformly to both Lie-type and Malcev-type brackets. This provides a unified spectral-analytic framework for two major strands of non-associative algebra.

The categorical perspective introduced here—centered on Hom--Lie morphisms and invariant subalgebras—lays a solid foundation for future investigations into representations, cohomology, and deformation quantization of twisted non-associative systems. Furthermore, by clearly identifying the current limitations (namely, the boundedness of the derivation and the requirement of norm-almost periodicity), our work opens concrete avenues for generalization. These include extensions to locally convex spaces and unbounded twisted derivations (via the Hille--Yosida theorem and Gevrey regularity), as well as the use of weaker notions of almost periodicity (such as Besicovitch spaces), with promising applications in quantum ergodicity, non-commutative geometry, and the spectral theory of quasi-periodic quantum systems.

\end{document}